\documentclass[11pt]{article}
\usepackage[all]{xy}
\usepackage{amsmath}
\usepackage{amsfonts}
\usepackage{amssymb}
\usepackage{amscd}
\usepackage{amsthm}
\usepackage{latexsym}
\usepackage{amsbsy}
\newtheorem{lemma}{Lemma}[section]
\newtheorem{proposition}{Proposition}[section]
\newtheorem{theorem}{Theorem}[section]
\newtheorem{corollary}{Corollary}[section]
\newtheorem{definition}{Definition}[section]

\newcommand{\nn}{\nonumber}
\newcommand{\C}{\mathbb{C}}
\newcommand{\qp}{\mathbb{C}P}
\newcommand{\ud}{\,\mathrm{d}}
\newcommand{\delb}{\overline{\partial}}
\newcommand{\del}{\partial}
\begin{document}
\title{Noncommutative complex geometry of  \\ the quantum projective space}
\author {Masoud Khalkhali and Ali Moatadelro\\
 Department of Mathematics, University of Western Ontario\\
 London, Ontario, Canada}
\date{}
\maketitle{}
\begin{abstract}
We define holomorphic structures on canonical line bundles of the quantum projective space $\qp^{\ell}_q$ and identify their space of holomorphic sections. This determines the quantum homogeneous coordinate ring of the quantum projective space. We show that the fundamental class of $\qp^{\ell}_q$ is naturally presented by a twisted positive Hochschild cocycle. Finally, we verify the main statements of Riemann-Roch formula and Serre duality for $\qp^{1}_q$ and $\qp^{2}_q$.
\end{abstract}
\begin{center}\tableofcontents
\end{center}
\section{Introduction}
In this paper we continue a study of complex structures on quantum projective
spaces that was initiated in \cite{KLS} for
$\qp^1_q$ and was further continued in \cite{KM} for the case $\qp^2_q$.
In the present paper we consider a natural holomorphic structure on the quantum
projective space  $\qp^{\ell}_q$ already presented in \cite{DD, DL},
and define  holomorphic structures on canonical quantum line bundles on it.
The space of holomorphic sections of these line bundles then will determine
the quantum homogeneous coordinate ring of $\qp^{\ell}_q$.

In Section \ref{prelim}, we review the  preliminaries on irreducible representations
of quantum groups $U_q(\mathfrak{su}(\ell +1))$  and the
Gelfand-Tsetlin   basis for these representations.
In Section \ref{complex} we recall the definition of
a complex structure, holomorphic line bundles and bimodule connections.
In Section \ref{cpl} we recall the definition of the quantum projective space
$\qp^{\ell}_q$, and endow its canonical line bundles with holomorphic connections. We also identify
the space of holomorphic sections of these line bundles. In Section
\ref{bimod}, we define bimodule connections on canonical line bundles.
This enables us to define the quantum homogeneous coordinate ring of
$\qp^{\ell}_q$ and identify this ring with the ring of twisted
polynomials. In Section \ref{positive} we introduce a twisted positive
Hochschild cocycle $2 \ell$-cocycle on $\qp^{\ell}_q$,  by using the complex structure of
 $\qp^{\ell}_q$, and show that it is  cohomologous to its fundamental class which is represented
by a  twisted cyclic cocycle. This certainly provides further evidence
for the belief, advocated by Alain Connes
 \cite{C1,C2}, that holomorphic structures in noncommutative geometry  should be  represented by (extremal)
 positive Hochschild cocycles within the fundamental class.
Finally in the last Section we verify directly that the main statements of Riemann-Roch formula and
Serre duality theorem hold for  that the $\qp^1_q$ and $\qp^2_q$.  

\section{Preliminaries on $U_q(\mathfrak{su}(\ell +1))$ and $\mathcal A(SU_q(\ell+1))$}
\label{prelim}
\subsection{The quantum enveloping algebra $U_q(\mathfrak{su}(\ell+1))$}
Let $0<q<1$. We use the following notation
\begin{align*}
&[a,b]_q=ab-q^{-1}ba,\quad[z]=\frac{q^z-q^{-z}}{q-q^{-1}},\quad
[n]!=[n][n-1]\cdots[1],\\ &\begin{bmatrix}n\\m\end{bmatrix}=\frac{[n]!}{[m]![n-m]!},\quad
[j_1,j_2,\cdots,j_k]!=q^{-\sum_{r<s}j_rj_s}\frac{[j_1+j_2+\cdots+j_k]!}{[j_1]![j_2]!\cdots[j_k]!}\,.
\end{align*}
The quantum enveloping algebra $U_q(\mathfrak{su}(\ell+1))$, as a $*$-algebra, is generated by elements $K_i,K_i^{-1},E_i,F_i$, $i=1,2,\cdots,\ell$, with $K_i^*=K_i$ and $E_i^*=F_i$, subject to the following relations for $0\leq i,j\leq \ell$ \cite{KS},
\begin{align}\label{EFK}
&K_iK_j=K_jK_i\qquad E_iK_i=q^{-1}K_iE_i\nn\\
&E_iK_j=q^{1/2}K_jE_i \qquad\text{if}\quad |i-j|=1\nn\\
&E_iK_j=K_jE_i \qquad\text{if}\quad |i-j|>1\\
&E_iF_j-F_jE_i=\delta_{ij}\frac{K_i^2-K_i^{-2}}{q-q^{-1}}\nn\\
&E_iE_j=E_jE_i \qquad\text{if}\quad |i-j|>1,\nn
\end{align}
and the Serre relation
\[
E_i^2E_j-(q+q^{-1})E_iE_jE_i+E_jE_i^2=0\qquad\text{if}\quad |i-j|=1.
\]
The coproduct, counit and antipode of this Hopf algebra is given by
\begin{align*}
&\Delta(K_i)=K_i\otimes K_i,\quad \Delta(E_i)=E_i\otimes K_i+K_i^{-1}\otimes E_i,\\
&\epsilon(K_i)=1,\quad\epsilon(E_i)=0,\quad S(K_i)=K_i^{-1},\quad S(E_i)=-qE_i.
\end{align*}
\subsection{The quantum group $\mathcal A(SU_q(\ell+1))$}
As a $*$-algebra, $\mathcal A(SU_q(\ell+1))$ is generated by $(\ell+1)^2$ elements $u^i_j$, where $i,j=1,2,...,\ell+1$ subject to the following commutation relations
\begin{align*}
&u^i_ku^j_k=qu^j_ku^i_k, &u^k_iu^k_j=qu^k_ju^k_i \quad \forall\, i<j,&\\
&[u^i_l,u^j_k]=0, &[u^i_k,u^j_l]=(q-q^{-1})u^i_lu^j_k &\quad \forall\, i<j,\,k<l,
\end{align*}
and
\begin{equation*}
\sum_{\pi\in S_{\ell+1}}(-q)^{||\pi||}u^1_{\pi(1)}u^2_{\pi(2)}\cdots u^{\ell+1}_{\pi(\ell+1)}=1,
\end{equation*}
where the sum is taken over all permutations of the $\ell+1$ elements and $||\pi||$ is the number of simple inversions of the permutation $\pi$. The involution is given by
\begin{align*}
(u^i_j)^*=(-q)^{j-i}\sum_{\pi\in S_{\ell}}(-q)^{||\pi||}u^{k_1}_{\pi(n_1)}u^{k_2}_{\pi(n_2)}\cdots u^{k_{\ell}}_{\pi(n_{\ell})}
\end{align*}
with $\{k_1,\cdots,k_{\ell}\}=\{1,2,\cdots,\ell+1\}\setminus \{i\}$ and $\{n_1,\cdots,n_{\ell}\}=\{1,2,\cdots,\ell+1\}\setminus \{j\}$ as ordered sets, and the sum is over all permutations $\pi$ of the set $\{n_1,\cdots,n_{\ell}\}$. The Hopf algebra structure is given by
\begin{align*}
\Delta(u^i_j)=\sum_k u^i_k\otimes u^k_j,\quad \epsilon(u^i_j)=\delta^i_j,\quad
S(u^i_j)=(u^j_i)^*.
\end{align*}
\subsection{Irreducible representations of $U_q(\mathfrak{su}(\ell+1))$ and the related Gelfand-Tsetlin tableaux}
The finite dimensional irreducible $*$-representations of $U_q(\mathfrak{su}(\ell+1))$ are indexed by $\ell-$tuples of non-negative integers $n:=(n_1,n_2,...,n_{\ell})$. We denote this representation by $V_n$. A basis for $V_n$ is given
by Gelfand-Tsetlin (GT) tableaux  that we denote it here by
\begin{align*}
|\underline m\rangle:=\begin{bmatrix}
m_{1,\ell+1} & m_{2,\ell+1} & \ldots & m_{\ell,\ell+1} &  m_{\ell+1,\ell+1}\\
m_{1,\ell} & m_{2,\ell} & \ldots & m_{\ell,\ell} \\
\vdots & \vdots\\
m_{1,2} & m_{2,2}\\
m_{1,1}
\end{bmatrix}
\end{align*}
where $n_i=m_{i,\ell+1}-m_{i+1,\ell+1}$ for $i=1,2,...,\ell$, which fixes $m_{ij}$ up to an additive constant. The action of generators on this basis is given by \cite{KS},
$K_k|\underline m\rangle=q^{\frac{a_k}{2}}|\underline m\rangle$, where
\begin{align}\label{ak}
a_k&=\sum_{i=1}^km_{i,k}-\sum_{i=1}^{k-1}m_{i,k-1}-\sum_{i=1}^{k+1}m_{i,k+1}+\sum_{i=1}^{k}m_{i,k}\\
&=2\sum_{i=1}^km_{i,k}-\sum_{i=1}^{k-1}m_{i,k-1}-\sum_{i=1}^{k+1}m_{i,k+1}\nn,
\end{align}
and the action of $E_k$ is given by
\begin{equation}\label{E}
E_k|\underline m\rangle=\sum_{j=1}^kA^j_k|\underline m^j_k\rangle,
\end{equation}
where $|\underline m^j_k\rangle$ is obtained from $|\underline m\rangle$ when $m_{j,k}$ is replaced by $m_{j,k}+1$ and
\begin{align}\label{Ajk}
A^j_k=\Big(-\frac
{\Pi_{i=1}^{k+1}[l_{i,k+1}-l_{j,k}]\Pi_{i=1}^{k-1}[l_{i,k-1}-l_{j,k}-1]}{
\Pi_{i\neq j}[l_{i,k}-l_{j,k}][l_{i,k}-l_{j,k}-1]}\Big)^{1/2}.
\end{align}
Here $l_{i,j}=m_{i,j}-i$, and the positive square root is taken.
For the inner product $\langle\underline i|\underline j\rangle:=\delta_{\underline i,\underline j}$ this will be a $*$-representation and the matrix coefficients of $\rho^n:U_q(\mathfrak{su}(\ell+1))\rightarrow End(V_n)$ will be $\rho^n_{\underline i,\underline j}(h)=\langle\underline i|h|\underline j\rangle$. Note that the basic representation of $U_q(\mathfrak{su}(\ell+1))$ is given by $\sigma:U_q(\mathfrak{su}(\ell+1))\rightarrow M_{\ell+1}(\mathbb{C})$ where
\begin{equation*}
\sigma^i_j(K_r)=\delta^i_jq^{\frac{1}{2}(\delta_{r+1,i}-\delta_{r,i})}, \quad \sigma^i_j(E_r)=\delta^i_{r+1}\delta^r_j,
\end{equation*}
and the Hopf pairing $\langle\,,\rangle:U_q(\mathfrak{su}(\ell+1))\times\mathcal A(SU_q(\ell+1))\rightarrow \mathbb{C}$ is defined by
$\langle h, u^i_j\rangle:=\sigma^i_j(h)$. Therefore
\begin{align}\label{pairing}
&\langle K_r,u^i_j\rangle=\sigma^i_j(K_r)=\delta^i_jq^{\frac{1}{2}(\delta_{r+1,i}-\delta_{r,i})},\nonumber\\
&\langle E_r,u^i_j\rangle=\sigma^i_j(E_r)=\delta^i_{r+1}\delta^r_j.
\end{align}
Using Peter-Weyl theorem, a basis $\{t^n_{\underline i,\underline j}\}$ for $\mathcal A(SU_q(\ell+1))$ is implicitly given by $\langle h,t^n_{\underline i,\underline j}\rangle=\rho^n_{\underline i,\underline j}(h)$. For later use it is worth mentioning here that for $n=(0,0,...,0,1)$ these basis $ t^n_{\underline i,\underline j}$ are just generators $u^i_j$. In order to show this, it is enough to compute $\rho^n_{\underline i,\underline j}(h)$ for generators of $U_q(\mathfrak{su}(\ell+1))$.
Indeed for $n=(0,0,...,0,1)$ a basis element $|\underline m\rangle$ takes the following form
\begin{align*}
|\underline m\rangle:=\begin{bmatrix}
m & m & \ldots & m & m &  m-1\\
m & m & \ldots & m & m_l \\
\vdots & \vdots\\
m & m_{2}\\
m_{1}
\end{bmatrix}
\end{align*}
where each of the $m_i$'s is either $m$ or $m-1$ such that $m_1\geq m_{2}\geq...\geq m_{l}$. So $|\underline m\rangle$ can be parametrized just by one integer $i$. Let us denote $|\underline m\rangle$ by $|i\rangle$ when $m_j=m$ for $j\leq i-1$ and $m_j=m-1$ for $j\geq i$.
\begin{align*}
\rho^n_{\underline i,\underline j}(K_r)=\langle\underline i|K_r|\underline j\rangle=q^{\frac{a_r}{2}}\langle\underline i|\underline j\rangle=q^{\frac{a_r}{2}}\delta_{\underline i,\underline j}.
\end{align*}
where
\[
a_r
=2\sum_{i=1}^rm_{i,r}-\sum_{i=1}^{r-1}m_{i,r-1}-\sum_{i=1}^{r+1}m_{i,r+1}.
\]
So for our case we will end up with
\begin{align*}
\rho^n_{i, j}(K_r)=\langle i|K_r| j\rangle=q^{\alpha/2}\langle i|j\rangle=q^{\alpha/2}\delta_{i,j},
\end{align*}
where
\begin{align*}
\alpha=\begin{cases}0\quad \text{if}\quad r\geq j \quad\text{or} \quad r\leq j-2\\
                   1 \quad \text{if} \quad r=j-1\\
                   -1 \quad \text{if} \quad r=j.
\end{cases}
\end{align*}
One can easily see that $\alpha=\delta_{r+1,j}-\delta_{r,j}$ and we get the same answer as (\ref{pairing}). Also we have
\begin{align*}
\rho^n_{i, j}(E_r)=\langle i|E_r| j\rangle=\delta^r_j\langle i|r+1\rangle=\delta_{i,r+1}\delta^r_j=\langle E_r,u^i_j\rangle,
\end{align*}
which can be obtained from (\ref{E}) and (\ref {Ajk}) since
\begin{align*}
E_r|r\rangle= A^r_r|r+1\rangle
\end{align*}
and
\begin{align*}
A^r_r=\Big(-\frac
{\Pi_{i=1}^{r+1}[l_{i,r+1}-l_{r,r}]\Pi_{i=1}^{r-1}[l_{i,r-1}-l_{r,r}-1]}{
\Pi_{i\neq j}[l_{i,r}-l_{r,r}][l_{i,r}-l_{r,r}-1]}\Big)^{1/2}.
\end{align*}
The fact that only $A^r_r$ contributes in the summation (\ref{E}) is simply because of the form of $|r\rangle$. We also have $E_r|j\rangle=0$ if $j\neq r$.
The value of this fraction is one since both the numerator and the denominator are equal to $[r]![r-1]!$. In particular we have $t^n_{\ell+1,j}=u^{\ell+1}_j=z_j$, the generators of the quantum sphere $\mathcal A(S^{2\ell+1}_q)$ to be defined in the next section.
\section{The complex structure of $\qp^{\ell}_q$}
\label{complex}
In this section we first review the general setup of a noncommutative complex structure on
a given $\ast$-algebra as introduced in \cite{KLS}. Then we shall define
a complex structure on $\qp^{\ell}_q$ and its canonical line bundles
following closely \cite{DD}.
\subsection{Noncommutative complex structures}
Let $\mathcal A$ be a $\ast$-algebra over $\mathbb C$. A \textit{differential $\ast$-calculus} for
 $\mathcal A$ is a pair $(\Omega^{\bullet}(\mathcal A),\ud)$, where
$\Omega^{\bullet}(\mathcal A)=\bigoplus_{n \geq 0}\Omega^n(\mathcal A)$
 is a graded differential $\ast$-algebra with $\Omega^0(\mathcal A)=\mathcal A$. The differential map
 $\ud:\Omega^{\bullet}(\mathcal A)\rightarrow \Omega^{\bullet +1}(\mathcal A)$ satisfies the graded Leibniz rule,
 $\ud(\omega_1\omega_2)=(\ud\omega_1)\omega_2+(-1)^{deg (\omega_1)}\omega_1(\ud\omega_2)$ and
 $\ud^2=0$. The differential also  commutes with the $\ast$-structure:
 $\ud(a^*)=(\ud a)^*$.
\begin{definition}
A complex structure on an algebra $\mathcal A$, equipped with a differential calculus
 $(\Omega^{\bullet}(\mathcal A),\ud)$,
is a bigraded differential $\ast$-algebra $\Omega^{(\bullet,\bullet)}(\mathcal A)$
and two differential
maps $\partial:\Omega^{(p,q)}(\mathcal A)\rightarrow \Omega^{(p+1,q)}(\mathcal A)$ and
$\delb:\Omega^{(p,q)}(\mathcal A)\rightarrow \Omega^{(p,q+1)}(\mathcal A)$ such that:
\begin{eqnarray}
\Omega^n(\mathcal A)=\bigoplus_{p+q=n}\Omega^{(p,q)}(\mathcal A)\,,\quad \partial a^*=(\delb a)^*\,,
 \quad \ud=\partial+\delb.
\end{eqnarray}
Also, the involution $\ast$ maps $\Omega^{(p,q)}(\mathcal A)$ to $\Omega^{(q,p)}(\mathcal A)$.
\end{definition}
We will use the simple notation $(\mathcal A,\delb)$ for a complex structure on $\mathcal A$.
\begin{definition}
Let $(\mathcal A,\delb)$ be an algebra with a complex structure. The space of holomorphic
elements of $\mathcal A$ is defined as
\begin{equation}
\mathcal O(\mathcal A):=Ker\{\delb:\mathcal A\rightarrow\Omega^{(0,1)}(\mathcal A)\}.\nonumber
\end{equation}
\end{definition}
\subsection{Holomorphic connections}
Suppose we are given a differential calculus $(\Omega^{\bullet}(\mathcal A),\ud)$. We recall that a
connection on a left $\mathcal A$-module $\mathcal E$ for the differential
calculus $(\Omega^{\bullet}(\mathcal A),\ud)$ is a linear map
$\nabla:\mathcal E\rightarrow\Omega^1(\mathcal A)\otimes_{\mathcal A} \mathcal E$
with left Leibniz property:
\begin{equation}\label{left Leibniz}
\nabla(a\xi)=a\nabla \xi+\ud a \otimes_{\mathcal A} \xi,\quad \forall a\in \mathcal A,\,
\forall \xi \in \mathcal E.
\end{equation}
By the graded Leibniz rule, i.e.
\begin{equation}
\nabla(\omega\xi)=(-1)^n\omega \nabla \xi+\ud \omega \otimes_{\mathcal A} \xi,
 \quad \forall\omega \in \Omega^n(\mathcal A),\, \forall \xi \in \Omega(\mathcal A)\otimes_{\mathcal A} \mathcal E,
\end{equation}
this connection can be uniquely extended to a map, which will be denoted again by $\nabla$, $\nabla:
\Omega^{\bullet}(\mathcal A) \otimes_{\mathcal A} \mathcal E
\rightarrow \Omega^{\bullet+1}(\mathcal A) \otimes_{\mathcal A} \mathcal E$.

The curvature of such a connection is defined by $F_\nabla=\nabla \circ \nabla$.
One can show that, $F_\nabla$ is an element of
Hom$_\mathcal A (\mathcal E,\Omega^2(\mathcal A)\otimes_{\mathcal A} \mathcal E)$.
\begin{definition}\label{holomorphic vect bundle}
Suppose $(\mathcal A,\delb)$ is an algebra with a complex structure.
 A holomorphic structure on a left $\mathcal A$-module
$\mathcal E$ with respect to this complex structure is given by a linear map
$\nabla^{\delb}:\mathcal E \rightarrow \Omega^{(0,1)} \otimes_{\mathcal A} \mathcal E$ such that
\begin{eqnarray}
&&\nabla^{\delb}(a\xi)=a\nabla^{\delb}\xi+\delb a \otimes_{\mathcal A} \xi, \quad \forall
a \in \mathcal A,\, \forall \xi \in \mathcal E,
\end{eqnarray}
and such that $F_{\nabla^{\delb}}=(\nabla^{\delb})^2=0$.
\end{definition}
Such a connection will be called a flat $\delb$-connection. In the case which $\mathcal E$ is a finitely generated $\mathcal A$-module, $(\mathcal E,\nabla^{\delb})$ will be called a holomorphic vector bundle.

Associated to a flat $\delb$-connection, there exists a complex of vector spaces
\begin{equation}
0\rightarrow \mathcal E \rightarrow \Omega^{(0,1)} \otimes_{\mathcal A} \mathcal E
 \rightarrow \Omega^{(0,2)} \otimes_{\mathcal A} \mathcal E\rightarrow...
\end{equation}
Here $\nabla^{\delb}$ is extended to $\Omega^{(0,q)} \otimes_{\mathcal A} \mathcal E$
by the graded Leibniz rule. The zeroth cohomology group of this complex is called the space of holomorphic sections of
$\mathcal E$ and will be denoted by $H^0(\mathcal E,\nabla^{\delb})$.
\subsection{Holomorphic structures on bimodules}
\begin{definition}
Let $\mathcal A$ be an algebra with a differential calculus $(\Omega^{\bullet}(\mathcal A),\ud)$.
A bimodule connection on an $\mathcal A$-bimodule $\mathcal E$ is given by a connection $\nabla$
 which satisfies a left Leibniz rule as in formula ({\ref{left Leibniz}})
and a right $\sigma$-twisted Leibniz property with respect to a bimodule isomorphism
$\sigma:\mathcal E\otimes_{\mathcal A} \Omega^1(\mathcal A) \rightarrow
\Omega^1(\mathcal A)\otimes_{\mathcal A} \mathcal E$. i.e.
 \begin{equation}
\nabla(\xi a)=(\nabla \xi) a+\sigma(\xi\otimes \ud a)\,, \quad\forall \xi \in\mathcal E,\,
\forall a \in \mathcal A.
\end{equation}
\end{definition}
The tensor product connection of two bimodule connections $\nabla_1$ and $\nabla_2$
on two $\mathcal A$-bimodules $\mathcal E_1$ and $\mathcal E_2$ with respect to the bimodule isomorphisms
$\sigma_1$ and $\sigma_2$ is a map $\nabla:\mathcal E_1\otimes_{\mathcal A} \mathcal E_2\rightarrow
\Omega^1(\mathcal A)\otimes_{\mathcal A}\mathcal E_1\otimes_{\mathcal A} \mathcal E_2$ defined by
\begin{equation}
\nabla:=\nabla_1 \otimes 1 +(\sigma_1\otimes 1)(1\otimes \nabla_2).\nonumber
\end{equation}
It can be checked that, $\nabla$ has the right $\sigma$-twisted property with
$\sigma:\mathcal E_1\otimes \mathcal E_2\otimes \Omega^1(\mathcal A)\rightarrow
\Omega^1(\mathcal A)\otimes\mathcal E_1\otimes \mathcal E_2$ given by
$\sigma=(\sigma_1\otimes 1)\circ(1\otimes\sigma_2)$.

\section{$\qp^{\ell}_q$ and the associated quantum line bundles}
\label{cpl}
We recall  the definition of the quantum projective space $\qp^{\ell}_q$ as
the quantum homogeneous space of the quantum group $SU_q(\ell+1)$ and its quantum
subgroup $U_q(\ell)$ from \cite{DD}. Let $\hat K:=(K_1K_2^2\cdots K_{\ell}^{\ell})^{2/{\ell+1}}$ and $\mathcal L_h a:=a\triangleleft S^{-1}(h)$. Then we define the quantum $2\ell+1$ sphere as
\[
\mathcal A(S^{2\ell+1}_q):=\{a\in\mathcal A(SU_q(\ell+1)) | \, \mathcal L_h(a)=\epsilon(h)a, \quad\forall h\in U_q(\mathfrak{su}(\ell))\}.
\]
The invariant elements of this space under the action of $\hat K$ will provide the coordinate functions of the quantum projective space
\[\label{qp}
\mathcal A(\qp^{\ell}_q):=\{a\in\mathcal A(S^{2\ell+1}_q) |\mathcal L_{\hat{K}}a=a\}.
\]
The space of sections of the canonical line bundles $L_N$, $N\in \mathbb Z$, are defined by
\begin{equation}\label{LN}
L_N:=\{a\in\mathcal A(S^{2\ell+1}_q) |\mathcal L_{\hat{K}}a=q^{\frac{N\ell}{\ell+1}}a\}.
\end{equation}
Let
\begin{align*}
M_{jk}:=[E_j,[E_{j+1},...,[E_{k-1},E_k]_q...]_q]_q\quad \text{for} \,1\leq j<k \leq \ell,
\end{align*}
and
\begin{align*}
N_{jk}:=(K_jK_{j+1}...K_{\ell}).(K_{k+1}K_{k+2}...K_{\ell}).\hat K^{-1}\quad \text{for} \,1\leq j<k \leq \ell.
\end{align*}
Let $X_i:=N_{i\ell}M^*_{i\ell}$ for $i=1,...,\ell$.
We will also use a right black action instead of left action by $h\blacktriangleright a:=a\triangleleft \theta(h)$, where $\theta:U_q(\mathfrak{su}(\ell+1))\rightarrow U_q(\mathfrak{su}(\ell+1))^{op}$ is the Hopf $*$-algebra isomorphism which is defined on generators as
\begin{equation}
\theta(K_i)=K_i,\quad\theta(E_i)=F_i,\quad\theta(F_i)=E_i,\nonumber
\end{equation}
and satisfying $\theta^2=id$.

For any  $r$-dimensional $*$-representation of $U_q(\mathfrak{u}(\ell))$ like $\sigma$, we define the $\mathcal A(\qp^{\ell}_q)$-bimodule $\mathfrak M(\sigma):=\{v\in \mathcal A(SU_q(\ell+1))^r\,|\,v\triangleleft h=\sigma(h)v,\, \forall h \in U_q(\mathfrak{u}(\ell))\}$ \cite {DD,DDL}. Suppose that $\sigma ^N_1$ is obtained from the basic representation $\sigma_1:U_q(\mathfrak{su}(\ell))\rightarrow\text{End}(\C^\ell)$ lifted to a representation of $U_q(\mathfrak{u}(\ell))$ by $\sigma^N_1(\hat K)=q^{1-\frac{\ell N}{\ell+1}}Id_{\C^\ell}$. Then the space of anti-holomorphic 1-forms is  given by $\Omega^{(0,1)}:=\mathfrak M(\sigma^0_1)$. Hence, any anti-holomorphic 1-form is a $\ell$-tuple $v:=(v_1,...,v_{\ell})$ such that $v\triangleleft h=\sigma^0_1(h)v$.
The complex structure of $\qp^\ell_q$ is given by
\begin{equation*}
\delb:=\sum \mathcal L_{\hat KX_i}\otimes \mathfrak e^L_{e^i}.
\end{equation*}
Here $e_i$'s are elements of the standard basis and $\mathfrak e^L_{e^i}$ is the left exterior product by $e_i$.
We show that on $\mathcal A(\qp^\ell_q)$ we have
\begin{align}\label{delbar}
\delb a=-\Big(a\triangleleft F_\ell F_{\ell-1}...F_1,a\triangleleft F_\ell F_{\ell-1}...F_2,...,a\triangleleft F_\ell F_{\ell-1},a\triangleleft F_\ell\Big).
\end{align}
In fact,
\begin{align*}
\mathcal L_{X_i} a&= a \triangleleft S^{-1}\Big(\hat KK_\ell K_{\ell-1}\cdots K_i\hat K^{-1}[...[[F_{\ell},F_{\ell-1}]_q,F_{\ell-2}]_q,...,F_i]_q\Big)\\&=(-q^{-1})^{\ell-i}(-q)^{\ell-i+1} a\triangleleft F_\ell F_{\ell-1}\cdots F_i\hat K K_i^{-1}K_{i+1}^{-1}...K_\ell^{-1}\hat K^{-1}\\&=(-1)^{2(\ell-i)+1}a\triangleleft \hat K K_i^{-1}K_{i+1}^{-1}...K_\ell^{-1}\hat K^{-1}F_\ell F_{\ell-1}...F_i\\&=-a\triangleleft F_\ell F_{\ell-1}...F_i
\end{align*}

Here we used the commutation relations (\ref{EFK}). The only order of $F_j$'s in the commutators that takes part in computation is $ F_\ell F_{\ell-1}...F_i$ and others vanish because $a\triangleleft F_j=0$ for $j<\ell,\,a\in\mathcal A(\qp^\ell_q)$. Note that, all elements of $\mathcal A(\qp^\ell_q)$ are fixed under the of action of all $K_i$'s.

We would like to find a basis for the space of sections of the canonical quantum line bundles $L_N$. Note that $L_0=\mathcal A(\qp^{\ell}_q)$.
By (\ref {LN}), the conditions that must hold are as follows
\begin{align}\label{LN2}
&K_i\blacktriangleright a=a, \quad E_i\blacktriangleright a=F_i\blacktriangleright a=0, \quad i=1,2,...,\ell-1, \nonumber\\
&K_1K_2^2...K_\ell^\ell\blacktriangleright a=q^{N\ell/2}a.
\end{align}
\begin{proposition}
 For any non-negative integer N, equations (\ref {LN2}) force that $|\underline m\rangle$,
  as a second component of $a=|\underline m'\rangle\otimes |\underline m\rangle$, to be of the
  form
\begin{align*}
\begin{bmatrix}
m_{1,\ell+1} & m & \ldots & m &  2m-m_{1,\ell+1}-N\\
m & m & \ldots & m \\
\vdots & \vdots\\
m & m\\
m
\end{bmatrix}.
\end{align*}
\end{proposition}
\begin{proof}
$K_1\blacktriangleright a=a$ and $E_1\blacktriangleright a=0$ give the equality for $m_{11}=m_{12}=m_{22}$. We know that $
K_k\blacktriangleright a=q^{\frac{a_k}{2}}a$, where $a_k$ is given by (\ref{ak}). For instance $a_1=2m_{11}-m_{12}-m_{22}$ and $a_2=2(m_{12}+m_{22})-m_{11}-(m_{13}+m_{23}+m_{33})$ and so on.
By (\ref{E}) and (\ref{Ajk}) we have
\begin{align*}
E_1|\underline m\rangle&=\Big(-[m_{11}-m_{12}][m_{11}-m_{22}+1]\Big)^{1/2}|\underline m^1_1\rangle\\
E_2|\underline m\rangle&=\Big(\frac{[m_{13}-m_{12}][m_{23}-m_{12}-1][m_{33}-m_{12}-2][m_{12}-m_{11}+1]}
{[m_{12}-m_{22}+1][m_{12}-m_{22}+2]}\Big)^{\frac{1}{2}}|\underline m^1_2\rangle\\
&+\Big(\frac{[m_{13}-m_{22}+1][m_{23}-m_{22}][m_{33}-m_{22}-1][m_{11}-m_{22}]}
{[m_{12}-m_{22}+1][m_{12}-m_{22}]}\Big)^{\frac{1}{2}}|\underline m^2_2\rangle.
\end{align*}
and
\begin{align*}
F_2|\underline m\rangle&=\Big(\frac{[m_{13}-m_{12}+1][m_{23}-m_{12}][m_{33}-m_{12}-1][m_{11}-m_{12}-1]}
{[m_{12}-m_{22}][m_{12}-m_{22}+1]}\Big)^{\frac{1}{2}}|\underline m^{-1}_2\rangle\\
&+\Big(\frac{[m_{13}-m_{22}+2][m_{23}-m_{22}+1][m_{33}-m_{22}][m_{11}-m_{22}-2]}
{[m_{12}-m_{22}+1][m_{12}-m_{22}+2]}\Big)^{\frac{1}{2}}|\underline m^{-2}_2\rangle.
\end{align*}
Now it is not difficult to see that $
K_1 |\underline m\rangle=|\underline m\rangle$ and $ E_1 |\underline m\rangle=0$, imposing
\begin{align*}
2m_{11}-m_{12}-m_{22}=0,\\
m_{11}-m_{12}=0.
\end{align*}
So $m_{11}=m_{12}=m_{22}$. In the same manner  $
K_2 |\underline m\rangle=|\underline m\rangle$, $ E_2 |\underline m\rangle=0$ and $ F_2 |\underline m\rangle=0$ give
\begin{align*}
2m_{12}+2m_{22}-m_{11}-m_{13}-m_{23}-m_{33}&=0,\\
m_{12}-m_{13}&=0,\\
m_{22}-m_{33}&=0.
\end{align*}
So we have $m_{11}=m_{12}=m_{22}=m_{13}=m_{23}=m_{33}$. Suppose that rows $1$ to $k$ with $k+1<\ell+1$, have been found equal to $m$. Let us prove that $ E_k |\underline m\rangle=0$ and $ F_k |\underline m\rangle=0$ will make the equality of all elements up to and including row $k+1$.
First note that in row $k+1$, we have $m_{2,k+1}=...=m_{k,k+1}=m$. Let us look at $A^1_k$.
\begin{align*}
A^1_k&=\Big(-\frac
{\Pi_{i=1}^{k+1}[l_{i,k+1}-l_{1,k}]\Pi_{i=1}^{k-1}[l_{i,k-1}-l_{1,k}-1]}{
\Pi_{i\neq j}[l_{i,k}-l_{1,k}][l_{i,k}-l_{1,k}-1]}\Big)^{1/2}\\
&=\Big(-\frac
{[l_{1,k+1}-l_{1,k}]...[l_{k+1,k+1}-l_{1,k}][l_{1,k-1}-l_{1,k}-1]...[l_{k-1,k-1}-l_{1,k}-1]}{
\Pi_{i\neq 1}[l_{i,k}-l_{1,k}][l_{i,k}-l_{1,k}-1]}\Big)^{1/2}.
\end{align*}
It is not hard to see that $A^1_k=0$ if $[l_{1,k+1}-l_{1,k}]=[m_{1,k+1}-m_{1,k}]=0$. So $m_{1,k+1}=m_{1,k}=m$ and by a similar observation the action of $F_k$ gives the equality $m_{k,k+1}=m_{kk}=m$.
 But to get to the very top row we need to use the action of $K_l$. We have
\begin{align*}
a_\ell&=2\sum_{i=1}^\ell m_{i,\ell}-\sum_{i=1}^{\ell-1}m_{i,\ell-1}-\sum_{i=1}^{\ell+1}m_{i,\ell+1}\\
&=2\ell m-(\ell-1)m-m_{1,\ell+1}-(\ell-1)m-m_{\ell+1,\ell+1}\\
&=2m-m_{1,\ell+1}-m_{\ell+1,\ell+1}.
\end{align*}
Since $\ell a_\ell/2=N\ell/2$, we see that $m_{l+1,l+1}=2m-m_{1,l+1}-N$.\\
\end{proof}
So we will find a basis for line bundles $ L_N$ as $\langle t^n_{\underline 0,\underline j} \rangle$,
where $n=(n_1,0,...,0,n_1+N)$ and
\begin{align*}
&|\underline 0\rangle=\begin{bmatrix}
m_{1,\ell+1} & m & \ldots & m &  2m-m_{1,\ell+1}-N\\
m & m & \ldots & m \\
\vdots & \vdots\\
m & m\\
m
\end{bmatrix}.
 \end{align*}
Note that for a negative integer $N$, the basis elements of $L_N$ are of the form of $t^n_{\underline 0,\underline j} $, where $n=(n_1-N,0,...,0,n_1)$.
\begin{theorem}\label{kerEl}
Let N be a non-negative integer. Then dim Ker $E_\ell\Big|_{L_N}=\binom{N+\ell}{\ell}$ and dim Ker $E_\ell\Big|_{L_{-N}}=0$.
\end{theorem}
\begin{proof}
The proof for $L_{-N}$ is easy, so we just consider the case $L_{N}$. A similar argument as the previous proposition shows that the vanishing of the action of $E_\ell$ on $|\underline m\rangle$ gives $m_{1,\ell+1}=m_{1,\ell}=m$. So we get the required
tableaux form. Now we count the free entries in the first component $|\underline m'\rangle$.
\begin{align*}
\begin{bmatrix}
m & m & \ldots & m &  m-N\\
x_{1,\ell} & x_{2,\ell} & \ldots & x_{\ell,\ell} \\
\vdots & \vdots\\
x_{1,2} & x_{2,2}\\
x_{1,1}
\end{bmatrix}
=
\begin{bmatrix}
m & m & \ldots & m &  m-N\\
m & m & \ldots & x_{\ell} \\
\vdots & \vdots\\
m & x_2\\
x_1
\end{bmatrix}
\end{align*}
with $x_i=x_{i,i}$. The question turns into a simple combinatorial problem of counting the number of non-decreasing sequences $m\geq x_1\geq x_2\geq ...\geq x_{\ell} \geq m-N$, which is $\binom{N+l}{\ell}$.
\end{proof}
\begin{corollary}
There is no non-constant holomorphic polynomials in $\mathcal A(\qp^\ell_q)$.
\end{corollary}
\begin{proof}
By (\ref{delbar}) it is obvious that $\delb a=0$ iff $E_\ell\blacktriangleright a=0$. Now the previous lemma for $N=0$ gives the result.
\end{proof}
\subsection{Holomorphic line bundles}
An anti-holomorphic connection on the line bundle $L_N$ is given by
\begin{align*}
&\nabla_N^{\delb}:L_N\rightarrow \Omega^{(0,1)}\otimes_{\mathcal A(\qp^\ell_q)} L_N\\
&\nabla_N^{\delb}(\xi):=q^{-N}\Psi^{\dagger}_N\delb\Psi_N,
\end{align*}
where $\Psi_N$ is a column vector \cite{DD}, given by $\Psi_N:=(\psi^N_{j_1,...,j_{\ell+1}})$ with
\begin{align*}
\psi^N_{j_1,...,j_{\ell+1}}:=[j_1,...,j_{\ell+1}]!^{1/2}(z_1^{j_1}...z_{\ell+1}^{j_{\ell+1}})^*,\quad\forall \,j_1+...+j_{\ell+1}=N.
\end{align*}
This is a flat connection as can be verified directly like \cite{KM}. This gives us the following Dolbeault complex
\begin{align*}
0\rightarrow L_N \rightarrow \Omega^{(0,1)}\otimes_{\mathcal A(\qp^\ell_q)} L_N\rightarrow
\cdots \rightarrow
\Omega^{(0,\ell)}\otimes_{\mathcal A(\qp^\ell_q)}L_N\rightarrow 0\,.
\end{align*}
The structure of the zeroth cohomology group $H^0( L_N,\nabla^{\delb}_N)$ of this complex which is called the
\textit {space of holomorphic sections of $L_N$},
is best described by the following theorem.
 \begin{corollary}
For any positive integer $N$, the space of holomorphic sections of the canonical line bundles of $\qp^\ell_q$ is
\begin{align*}
&H^0(L_N,\nabla_N)\simeq\mathbb{C}^{\binom{N+\ell}{\ell}} ,\\
&H^0(L_{-N},\nabla_{-N})=0.
\end{align*}
\end{corollary}
\begin{proof}
It is not difficult to see that the kernel of $\nabla_N^{\delb}$ coincides with the kernel of $E_\ell\blacktriangleright (.)$. Now the result is an obvious consequence of (\ref{kerEl}).
\end{proof}
Here we would like to establish the fact that for any integers $N$ and $M$ we have a bimodule isomorphism$L_N\otimes_{\mathcal A(\qp^{\ell}_q)} L_M\simeq L_{N+M}$. The multiplication map from left to right is an injective $\mathcal A(\qp^{\ell}_q)$-bilinear map. To see that this map is a surjection, we use a $PBW$- basis for $\mathcal A(S^{2\ell+1}_q)$ generated by \[\{z_1^{s_1}z_2^{s_2}\cdots z_{\ell}^{s_{\ell}}(z_1^*)^{t_1}(z_2^*)^{t_2}\cdots (z_{\ell-1}^*)^{t_{\ell-1}},\,z_1^{s_1}z_2^{s_2}\cdots z_{\ell-1}^{s_{\ell-1}}(z_1^*)^{t_1}(z_2^*)^{t_2}\cdots (z_{\ell}^*)^{t_{\ell}}\},
\]
for non-negative integers $s_i$ and $t_i$. Since 
\[K_j\blacktriangleright z_i=z_i, \quad K_j\blacktriangleright z_i^*=z_i^* \, for \,j<\ell\]
and 
\[K_{\ell}\blacktriangleright z_i=q^{1/2}z_i,\quad K_\ell\blacktriangleright z_i^*=q^{-1/2}z_i^*,\]
we have 
\[
 K_1K_2^2\cdots K_\ell^\ell\blacktriangleright Z
=q^{\ell/2\{\sum s_i-\sum t_i\}}Z,
\]
where
\[
Z=z_1^{s_1}z_2^{s_2}\cdots z_{\ell}^{s_{\ell}}(z_1^*)^{t_1}(z_2^*)^{t_2}\cdots (z_{\ell-1}^*)^{t_{\ell-1}}
\]
or
\[
Z=z_1^{s_1}z_2^{s_2}\cdots z_{\ell-1}^{s_{\ell-1}}(z_1^*)^{t_1}(z_2^*)^{t_2}\cdots (z_{\ell}^*)^{t_{\ell}}.
\]
It is obvious that $ Z\in L_N$ iff $\sum s_i-\sum t_i=N$. 
 
Now suppose that $Z=z_1^{s_1}z_2^{s_2}\cdots z_{\ell}^{s_{\ell}}(z_1^*)^{t_1}(z_2^*)^{t_2}\cdots (z_{\ell-1}^*)^{t_{\ell-1}}\in L_{N+M}$ and suppose $k$ is the first positive integer such that $\sum_{i=1}^ks_i>N$. Then take a partition of $N$ as $\sum_{i=1}^k r_i=N$, such that $s_i-r_i\geq 0$. Now the following is a preimage of $Z$.
\[
q^R  \big(z_1^{r_1}z_2^{r_2}\cdots z_{k}^{r_k}\otimes z_1^{s_1-r_1}z_2^{s_2-r_2}\cdots z_k^{s_k-r_k}z_{k+1}^{r_{k+1}}\cdots z_{\ell}^{s_{\ell}}(z_1^*)^{t_1}(z_2^*)^{t_2}\cdots (z_{\ell-1}^*)^{t_{\ell-1}}\big),
\]
where 
\[
R=r_k\{(s_{k-1}-r_{k-1})+\cdots+(s_1-r_1)\}+r_{k-1}\{(s_{k-2}-r_{k-2})+\cdots+(s_1-r_1)\}+
\cdots+r_2(s_1-r_1).
\]
By the above discussion it is obvious that $Z_1:=z_1^{r_1}z_2^{r_2}\cdots z_{k}^{r_k}\in L_N$ and  \[Z_2:=z_1^{s_1-r_1}z_2^{s_2-r_2}\cdots z_k^{s_k-r_k}z_{k+1}^{r_{k+1}}\cdots z_{\ell}^{s_{\ell}}(z_1^*)^{t_1}(z_2^*)^{t_2}\cdots (z_{\ell-1}^*)^{t_{\ell-1}}\in L_M.\] The result is obtained by noting that the product $Z_1Z_2=q^{-R}Z$.

For later use we would like to mention here that $\Omega^{(0,\ell)}\otimes_{\mathcal A(\qp^\ell_q)}L_N\simeq L_{\ell+1}\otimes_{\mathcal A(\qp^\ell_q)} L_N\simeq L_{N+\ell+1}$. In order to see this we recall the definition of $\Omega^{(0,\ell)}:=\mathfrak M(\sigma^0_\ell)$, where $\sigma ^0_k$ is obtained from the representation $\sigma_k:U_q(\mathfrak{su}(\ell))\rightarrow\text{End}(W_k)$ lifted to a representation of $U_q(\mathfrak{u}(l))$ by $\sigma^0_k(\hat K)=q^k Id_{W_k}$ \cite{DD}. We define the $\mathcal A(\qp^{\ell}_q)$-bimodule $\mathfrak M(\sigma):=\{v\in \mathcal A(SU_q(\ell+1))^r\,|\,v\triangleleft h=\sigma(h)v,\, \forall h \in U_q(\mathfrak{u}(\ell))\}$, where $\sigma$ is an $r$-dimensional $*$-representation of $U_q(\mathfrak{u}(\ell))$. So in our case $\sigma^0_\ell$ will be a $1$-dimensional $*$-representation of $U_q(\mathfrak{su}(\ell))$. Hence, any anti-holomorphic $\ell$-form is an element like  $v$ such that $v\triangleleft h=\sigma^0_\ell(h)v$.
The conditions that must hold are:
\begin{align*}
&K_i\blacktriangleright a=a, \quad E_i\blacktriangleright a=F_i\blacktriangleright a=0, \quad i=1,2,...,\ell-1.\\
&K_1K_2^2...K_\ell^\ell\blacktriangleright a=q^{\ell(\ell+1)/2}a.
\end{align*}
This gives us $\Omega^{(0,\ell)}\simeq L_{\ell+1}$.
\section{Bimodule connections}
\label{bimod}
In this section we would like to show that line bundles $L_N$ accept a bimodule connection in the sense of \cite{KLS}.
This means that there exists an isomorphism $\lambda_N:L_N\otimes_{\mathcal A(\qp^{\ell}_q)} \Omega^{(0,1)}\rightarrow \Omega^{(0,1)} \otimes_{\mathcal A(\qp^{\ell}_q)}L_N$ such that
\begin{equation*}
\nabla_N^{\delb}(\xi a):=(\nabla_N^{\delb}\xi)a+\lambda_N(\xi\otimes a).
\end{equation*}
Let us first check the case $N=1$. We will define $\lambda_1:=\alpha_1^{-1}\beta_1$ where,
\begin{align*}
&\alpha_1:\Omega^{(0,1)} \otimes_{\mathcal A(\qp^{\ell}_q)} L_1\rightarrow \mathcal A(SU_q(\ell+1))^{\ell},\\
&\alpha_1((v_1,...,v_{\ell})\otimes \xi):=q^{1/2}(v_1\xi,...,v_{\ell}\xi)
\end{align*}
and
\begin{align*}
&\beta_1:L_1\otimes_{\mathcal A(\qp^{\ell}_q)} \Omega^{(0,1)}\rightarrow \mathcal A(SU_q(\ell+1))^{\ell},\\
&\beta_1(\xi \otimes(v_1,...,v_{\ell})):=q^{-1/2}(\xi v_1,...,\xi v_{\ell}).
\end{align*}
Note that on the generators $p_{jk}:=z_j^*z_k$ of $\mathcal A(\qp^{\ell}_q)$ we will have
\begin{align*}
\delb p_{jk}=\Big((-1)^{\ell}q^{-1/2-(\ell-1)}(u^1_j)^*,...,q^{-3/2}(u^{\ell-1}_j)^*,-q^{-1/2}(u^{\ell}_j)^*\Big)u^{\ell+1}_k,
\end{align*}
and a basis element of $ L_1$ is of the form of $t^n_{\underline 0,\underline i}$, where $n=(n_1,0,...,0,n_1+1)$,
\begin{align*}
|\underline 0\rangle=\begin{bmatrix}
m_{1,\ell+1} & m & \ldots & m &  m_{1,\ell+1}-1\\
m & m & \ldots & m \\
\vdots & \vdots\\
m & m\\
m
\end{bmatrix}
\end{align*}
and
\begin{align*}
|\underline i\rangle=\begin{bmatrix}
m_{1,\ell+1} & m & \ldots m & m &  m_{1,\ell+1}-1\\
x_{\ell} & m & \ldots m & y_{\ell} \\
\vdots & \vdots\\
x_2 & y_2\\
x_1
\end{bmatrix}.
\end{align*}
Any typical element of $L_1\otimes_{\mathcal A(\qp^{l}_q)} \Omega^{(0,1)}$ is a linear combination of $t^n_{\underline 0,\underline i}\otimes p_{rs}\delb p_{jk}$.
We claim that
\[ v:=t^n_{\underline 0,\underline i}p_{rs}\Big((-1)^{\ell}q^{-1/2-(\ell-1)}(u^1_j)^*,...,q^{-3/2}(u^{\ell-1}_j)^*,-q^{-1/2}(u^{\ell}_j)^*\Big)\in \Omega^{(0,1)}
\]
and $u^{\ell+1}_k\in  L_1$. The latter can be easily obtained from the following observation,
\[
u^{\ell+1}_k\triangleleft (K_1K_2^2\cdots K_\ell^\ell)^{2/{\ell+1}}=q^{\ell/{\ell+1}}u^{\ell+1}_k.
\]
So we have to show that $v\triangleleft h=\sigma^0_1(h)v,$ for all $h \in U_q(\mathfrak{u}(\ell))$ . We will check this on generators. We need to know the following actions
\begin{align*}
t^n_{\underline 0,\underline i}\triangleleft h \quad \text{and} \quad (u^r_j)^*\triangleleft h
\end{align*}
for $h=E_i,F_i,K_i$ and $i=1,...,\ell-1$.

For $\sigma^0_1:U_q(\mathfrak{su}(\ell+1))\rightarrow \text{End}(W_1(\simeq\mathbb{C}^\ell))$ we have \cite{DD}
\begin{align*}
&(\sigma^0_1(K_r)w)_I=q^{1/2r\#I}w_I\\
&(\sigma^0_1(E_r)w)_I=\delta_{r\#I,+1}w_{I^{r,+}}\\
&(\sigma^0_1(F_r)w)_I=\delta_{r\#I,-1}w_{I^{r,-}}
\end{align*}
where $I=1,...,\ell$. Here $r\#I=1$ if $r=I$, $r\#I=-1$ if $r=I+1$ and $r\#I=0$ otherwise. For $K_1$ we have
\begin{align*}
\sigma^0_1(K_1)v&=\sigma^0_1(K_1)\Big(t^n_{\underline 0,\underline i}\Big((-1)^{\ell}q^{-1/2-(\ell-1)}(u^1_j)^*,...,q^{-3/2}(u^{\ell-1}_j)^*,-q^{-1/2}(u^{\ell}_j)^*\Big)\Big)\\
&=t^n_{\underline 0,\underline i}\Big((-1)^{\ell}q^{-(\ell-1)}(u^1_j)^*,(-1)^{\ell-1}q^{-1-(\ell-2)}(u^2_j)^*,...,-q^{-1/2}(u^{\ell}_j)^*\Big).
\end{align*}
Note that just a factor of $q^{1/2}$ and $q^{-1/2}$ contributed in the first and second component respectively.
But on the other hand
\begin{align*}
\Big\{t^n_{\underline 0,\underline i}&\Big((-1)^{\ell}q^{-1/2-(\ell-1)}(u^1_j)^*,...,q^{-3/2}(u^{\ell-1}_j)^*,-q^{-1/2}(u^{\ell}_j)^*\Big)\Big\}\triangleleft K_1\\
&=t^n_{\underline 0,\underline i}\triangleleft K_1\Big((-1)^{\ell}q^{-1/2-(\ell-1)}(u^1_j)^*,...,q^{-3/2}(u^{\ell-1}_j)^*,-q^{-1/2}(u^{\ell}_j)^*\Big)\triangleleft K_1\\
&=t^n_{\underline 0,\underline i}\Big((-1)^{\ell}q^{-(\ell-1)}(u^1_j)^*,(-1)^{\ell-1}q^{-1-(\ell-2)}(u^2_j)^*,...,-q^{-1/2}(u^{\ell}_j)^*\Big).
\end{align*}
Here we used the fact that $t^n_{\underline 0,\underline i}\triangleleft K_1=t^n_{\underline 0,\underline i}$ and
$u^i_j\triangleleft K_1=q^{1/2(\delta_{2,i}-\delta_{1,i})}u^i_j$, which follows from
\begin{align*}
u^i_j\triangleleft K_1=\sum_k\langle K_1,u^i_k\rangle u^k_j=\sum_k\delta^i_kq^{1/2(\delta_{2,i}-\delta_{1,i})}u^k_j=q^{1/2(\delta_{2,i}-\delta_{1,i})}u^i_j.
\end{align*}
Now we have
\begin{align*}
u^i_j\triangleleft K_1^{-1}=q^{-1/2(\delta_{2,i}-\delta_{1,i})}u^i_j
\end{align*}
and
\begin{align*}
(u^i_j\triangleleft K_1^{-1})^*=(u^i_j)^*\triangleleft S(K_1^{-1})^*=q^{-1/2(\delta_{2,i}-\delta_{1,i})}(u^i_j)^*.
\end{align*}
Therefore
\begin{align*}
(u^i_j)^*\triangleleft K_1=q^{-1/2(\delta_{2,i}-\delta_{1,i})}(u^i_j)^*.
\end{align*}
The same proof works for other $K_r$'s.

Now let us prove the equality for $E_1$.
\begin{align*}
\sigma^0_1(E_1)v&=\sigma^0_1(E_1)\Big\{t^n_{\underline 0,\underline i}\Big((-1)^{\ell}q^{-1/2-(\ell-1)}(u^1_j)^*,...,q^{-3/2}(u^{\ell-1}_j)^*,-q^{-1/2}(u^{\ell}_j)^*\Big)\Big\}\\
&=t^n_{\underline 0,\underline i}\Big((-1)^{\ell-1}q^{-1/2-(\ell-2)}(u^2_j)^*,0,0,...,0\Big)
\end{align*}
On the other hand
\begin{align*}
\Big\{t^n_{\underline 0,\underline i}&\Big((-1)^{\ell}q^{-1/2-(\ell-1)}(u^1_j)^*,...,q^{-3/2}(u^{\ell-1}_j)^*,-q^{-1/2}(u^{\ell}_j)^*\Big)\Big\}\triangleleft E_1\\
&=t^n_{\underline 0,\underline i}\triangleleft K_1^{-1}\Big((-1)^{\ell}q^{-1/2-(\ell-1)}(u^1_j)^*,...,q^{-3/2}(u^{\ell-1}_j)^*,-q^{-1/2}(u^{\ell}_j)^*\Big)\triangleleft E_1\\
&=t^n_{\underline 0,\underline i}\Big((-1)^{\ell-1}q^{-1/2-(\ell-2)}(u^2_j)^*(u^2_j)^*,0,0,...,0\Big)
\end{align*}
Note that
\begin{align*}
u^i_j\triangleleft E_1=\sum_k\langle E_1,u^i_k\rangle u^k_j=\sum_k\delta^i_2\delta^1_k u^k_j=\delta^i_2u^1_j.
\end{align*}
and $(u^i_j)^*\triangleleft E_1=-q\delta^i_1(u^2_j)^*$. The same argument will work for other $E_r$'s.
\begin{lemma}
With the above notation Im $\alpha_1=$Im $\beta_1$.
\end{lemma}
\begin{proof}
By the discussion before the lemma, the proof is clear.
\end{proof}
Taking the isomorphism $\lambda_1:=\alpha_1^{-1}\beta_1:L_1\otimes_{\mathcal A(\qp^{\ell}_q)} \Omega^{(0,1)}\rightarrow \Omega^{(0,1)} \otimes_{\mathcal A(\qp^{\ell}_q)}L_1$, since $L_N=L_1^{\otimes^N}$, we can define the isomorphism $\lambda_N:L_N\otimes_{\mathcal A(\qp^{\ell}_q)} \Omega^{(0,1)}\rightarrow \Omega^{(0,1)} \otimes_{\mathcal A(\qp^{\ell}_q)}L_N$ by
\[
\lambda_N:=(\lambda_1\otimes 1^{N-1})\circ(1\otimes\lambda_1\otimes 1^{N-2})\circ \cdots\circ(1^{N-1}\otimes\lambda_1).
\]

Now, we prove that $\nabla^{\delb}_N$ has the right $\lambda_N$-twisted Leibniz property.
\begin{proposition}
Taking $\lambda_N$ as above, the following holds
\begin{align*}\label{twisted Leibniz}
\nabla^{\delb}_N(\xi a)=(\nabla^{\delb}_N\xi) a+\lambda_N(\xi\otimes \delb a),
\quad \forall a \in \mathcal A(\qp^{\ell}_q),\quad\forall\xi \in L_N,
\end{align*}
i.e. $\nabla^{\delb}_N$ is a bimodule connection on $L_N$.
\end{proposition}
\begin{proof}
Let us compute first the last (i.e. the $\ell$ th component) of the left hand side. Since $\xi \triangleleft K_{\ell}=q^{N/2}\xi$ and $\Psi_N\triangleleft K_{\ell}=q^{-N/2}\Psi_N$, we will find that
\begin{align}
q^{-N}\Psi_N^{\dagger}((\Psi_N\xi a)\triangleleft F_{\ell})\nonumber
&=q^{-N}\Psi_N^{\dagger}\{(\Psi_N\triangleleft F_{\ell})((\xi a)\triangleleft K_{\ell})+(\Psi_N\triangleleft K_{\ell}^{-1})((\xi a)\triangleleft F_{\ell})\}\nonumber\\
&=q^{-N/2}(\xi \triangleleft F_{\ell})a+q^{-N}\xi (a\triangleleft F_{\ell}).\nonumber
\end{align}
For the $\ell$ th  component of the right hand we have
\begin{eqnarray}
q^{-N/2}(\xi \triangleleft F_{\ell}) a+\lambda_N(\xi \otimes a\triangleleft F_{\ell})\nonumber.
\end{eqnarray}
The previous lemma says that $q^{-N}$ will appear after acting by $\lambda_N$ on the second term.
It can be seen that $\alpha_N$ of both sides coincides. For $i$ th component of the left hand side we have
\begin{align*}
q^{-N}\Psi_N^{\dagger}&((\Psi_N\xi a)\triangleleft F_{\ell}F_{\ell-1}\cdots F_{i})\\
&=q^{-N}\Psi_N^{\dagger}\{(\Psi_N\triangleleft F_{\ell})((\xi a)\triangleleft K_{\ell})+(\Psi_N\triangleleft K_{\ell}^{-1})((\xi a)\triangleleft F_{\ell})\}F_{\ell-1}\cdots F_i\\
&=q^{-N/2}\{(\xi \triangleleft F_{\ell})a+q^{-N}\xi (a\triangleleft F_{\ell})\}F_{\ell-1}\cdots F_{i}\\
&=q^{-N/2}\{(\xi \triangleleft F_{\ell}\cdots F_{i})a+q^{-N}\xi (a\triangleleft F_{\ell}\cdots F_{i})\}.
\end{align*}
For the right hand side we get the same result. Other components will be computed similarly.
\end{proof}
Now we can prove that the two holomorphic structures on $ L_N\otimes_{\mathcal A(\qp^{\ell}_q)} L_M$ and $L_{N+M}$ are identical after the canonical isomorphism of these two spaces.
\begin{proposition}\label{Ln+m}
The tensor product connection
$\nabla_N^{\delb} \otimes 1+ (\lambda_N\otimes 1)(1 \otimes \nabla_M^{\delb})$
 coincides with the holomorphic structure on $L_N\otimes_{\mathcal A(\qp^{\ell}_q)} L_M$
when identified with $L_{N+M}$.
\end{proposition}
\begin{proof}
We will look at the last component first.
\begin{align*}
\Big\{\nabla_{N+M}^{\delb}(\xi_1\xi_2)\Big\}_{\ell}
&=q^{-(N+M)}\Psi^{\dagger}_{N+M}\delb\Psi_{N+M}(\xi_1\xi_2) \\
&=q^{-(N+M)}\Psi^{\dagger}_{N+M} (\Psi_{N+M}\xi_1\xi_2)
\triangleleft F_{\ell}\\
&=q^{-(N+M)}\Psi^{\dagger}_{N+M}
(\Psi_{N+M}\triangleleft F_{\ell}) ((\xi_1\xi_2)\triangleleft K_{\ell})\\
&+q^{-(N+M)}\Psi^{\dagger}_{N+M}(\Psi_{N+M}\triangleleft K_{\ell}^{-1})( (\xi_1\xi_2)\triangleleft F_{\ell})\\
&=q^{-\frac{N+M}{2}}(\xi_1\xi_2)\triangleleft F_{l}
\\&= q^{-\frac{N}{2}}(\xi_1 \triangleleft F_{\ell})\xi_2
+q^{-N-M/2}\xi_1(\xi_2 \triangleleft F_{\ell}).
\end{align*}

On the other hand
\begin{align*}
\Big\{((\nabla_N^{\delb}\otimes 1)+(\lambda_N \otimes 1)(1 \otimes
\nabla_M^{\delb}))(\xi_1\otimes \xi_2)\Big\}_{\ell}&=
q^{-N/2}\xi_1 \triangleleft F_{\ell}\otimes \xi_2\\
&+(\lambda_N\otimes 1)(\xi_1 \otimes q^{-M/2}
\xi_2 \triangleleft F_{\ell}).\nonumber
\end{align*}
Interpreting this expression as an element of $\Omega^{(0,1)}\otimes L_{N+M}$,
 after applying the map $\lambda_N$, which gives us $q^{-N}$ on the second summand,
 we will get the same result. The same argument as previous proposition gives the result for other components.
\end{proof}
Now the quantum homogeneous coordinate ring
$R:=\bigoplus_{\substack n\geq 0}H^0( L_N,\nabla_N^{\delb})$ of the quantum projective
space can be described as follows. This result was first obtained for $\ell =1, 2$ in \cite{KLS, KM}
where
its relation with the work in \cite{Av1, Av2} is also explained.
\begin{theorem}\label{thm ring structure}
We have the algebra isomorphism
\begin{align*}
\bigoplus_{\substack n\geq 0}H^0( L_N,\nabla_N^{\delb})\simeq
\frac{\mathbb{C} \langle z_1,z_2,...,z_{\ell}\rangle}{\langle\, z_iz_j-qz_jz_i:1\leq i<j\leq \ell\,\rangle}\nonumber
\end{align*}
\end{theorem}
\begin{proof}
The ring structure on $R$ is coming from the tensor product $L_{N_1}\otimes_{\mathcal A(\qp^{\ell}_q)}L_{N_2} \simeq L_{N_1+N_2}$. Since the basis elements $t^{(0,...,0,1)}_{0,j}$ of $H^0( L_1,\nabla _1^{\delb})$,  as shown in section 2 are $z_j$ for $j=1,2,...,\ell$, one can easily see that $H^0( L_1,\nabla_1^{\delb})=\mathbb Cz_1
\oplus\mathbb Cz_2\oplus\cdots\oplus\mathbb Cz_{\ell}$.
Now the isomorphism follows from the identities
$z_i\otimes_{\mathcal A(\qp^{\ell}_q)}z_j-qz_j\otimes_{\mathcal A(\qp^{\ell}_q)}z_i=0$ in $L_2$, which is obvious.\\
\end{proof}
\section{Existence of a twisted positive Hochschild cocycle for $\qp^\ell_q$}
\label{positive}
In \cite{C1},  Section VI.2, Connes shows
that extremal positive Hochschild cocycles in the sense of \cite{C2} on the algebra of smooth functions on a compact oriented
2-dimensional manifold encode the information needed to define a holomorphic
structure on the surface. There is
a similar result for holomorphic structures on the noncommutative two torus (cf. {\it Loc cit.}).
In particular the positive Hochschild cocycle is defined via the holomorphic structure
and represents the fundamental cyclic  cocycle. In \cite{KLS} a notion of twisted positive Hochschild cocycle is
 introduced
and a similar result is proved for the  holomorphic structure of $\qp^1_q$ and $\qp^2_q$ in \cite{KLS,KM}.
Although the corresponding problem of characterizing holomorphic
structures on higher dimensional (commutative or noncommutative) manifolds via positive Hochschild cocycles is still
open, nevertheless these results suggest regarding (twisted) positive Hochschild cocycles as a possible
framework for holomorphic noncommutative structures. In this section we
prove an analogous result for $\qp^\ell_q$ for all $\ell$.

First we recall the notion of twisted Hochschild and cyclic cohomologies.
Let $\mathcal A$ be an algebra and $\sigma$ an automorphism of $\mathcal A$.
For each $n\geq 0$, $C^n(\mathcal A):=$ Hom$(\mathcal A^{\otimes(n+1)},\mathbb C)$
is the space of \textit{n-cochains} on $\mathcal A$.
Define the space of \textit {twisted Hochschild n-cochains} as
$C^n_{\sigma}(\mathcal A):=$Ker$\{(1-\lambda_{\sigma}^{n+1}):C^n(\mathcal A)\rightarrow C^n(\mathcal A)\}$,
where the \textit{twisted cyclic} map $\lambda_{\sigma}:C^n(\mathcal A)\rightarrow C^n(\mathcal A)$ is defined as
\begin{eqnarray}
(\lambda_{\sigma}\phi)(a_0,a_1,...,a_n)=(-1)^n \phi(\sigma(a_n),a_0,a_1,...,a_{n-1}).\nonumber
\end{eqnarray}

The \textit {twisted Hochschild coboundary} map $b_{\sigma}:C^n(\mathcal A)\rightarrow C^{n+1}(\mathcal A)$ is given by
\begin{align*}
b_{\sigma}\phi(a_0,a_1,...,a_{n+1})=&
\sum_{i=0}^n(-1)^i\phi(a_0,...,a_ia_{i+1},...,a_{n+1})\\
&+(-1)^{n+1}\phi(\sigma(a_{n+1})a_0,...,a_n).\nonumber
\end{align*}
The cohomology of the complex $(C^*_{\sigma}(\mathcal A),b_{\sigma})$ is called
 the \textit
{twisted Hochschild cohomology} of $\mathcal A$.
We also need the notion of \textit {twisted cyclic cohomology} of $\mathcal A$.
It is by definition the cohomology of the complex
$(C^*_{\sigma,\lambda}(\mathcal A),b_{\sigma})$,
where
\begin{eqnarray}
C^n_{\sigma,\lambda}:=Ker\{(1-\lambda):C^n_{\sigma}(\mathcal A)\rightarrow
C^{n+1}_{\sigma}(\mathcal A)\}.\nonumber
\end{eqnarray}

Now we come back to the case of our interest, that is $\qp^\ell_q$.
Let $\tau$ be the fundamental class on $\mathbb{C}P^\ell_q$ defined as in \cite{DL} by a twisted cyclic cocycle
\begin{equation}\label{tau}
\tau(a_0,a_1,a_2,\cdots a_{2\ell}):=\int_h a_0 \ud a_1\ud a_2\cdots\ud a_{2\ell}\,,\quad\forall a_i \in \mathcal A(\qp^\ell_q).
\end{equation}
Here $h$ stands for the Haar state functional of the quantum group $\mathcal A(SU_q(\ell+1))$
which has a twisted tracial property $h(xy)=h(y\sigma(x))$. Here the algebra automorphism $\sigma$
is defined by
\begin{equation}
\sigma:\mathcal A(SU_q(\ell+1))\rightarrow \mathcal A(SU_q(\ell+1)), \quad \sigma(x)=K\triangleright x\triangleleft K.\nonumber
\end{equation}
where $K=(K_1^\ell K_2^{2(\ell-1)}\cdots K_j^{j(\ell-j+1)}\cdots K_\ell^\ell)^{2}$, see\cite{DD}. The map $\sigma$, restricted to the algebra $\mathcal A(\qp^\ell_q)$ is given by
$\sigma(x)=K\triangleright x$. Non-triviality of $\tau$ has been shown in \cite{DL}. Now we recall the definition of a twisted positive Hochschild cocycle as given in \cite{KLS}.
\begin{definition}
A twisted Hochschild 2n-cocycle $\phi$ on a $\ast$-algebra $\mathcal A$ is said to be
twisted positive if the following map defines a positive sesquilinear form on the vector space $\mathcal A^{\otimes(n+1)}$:
\begin{eqnarray}
\langle a_0\otimes a_1\otimes...\otimes a_n,b_0\otimes b_1\otimes...\otimes b_n\rangle=
\phi(\sigma(b_n^*)a_0,a_1,...,a_n,b_n^*,...,b_1^*).\nonumber
\end{eqnarray}
\end{definition}
\subsection{A twisted positive Hochschild cocycle on $\mathbb{C}P^\ell_q$.}
We recall that the set of $(\ell,\ell)$-shuffles (denoted by $S_{\ell,\ell}$) is set of all permutations $\pi\in S_{2\ell}$ such that $ \pi(1)<\pi(2)<\cdots<\pi(\ell)$ and $ \pi(\ell+1)<\pi(\ell+2)<\cdots<\pi(2\ell)$. Here we would like to look at a shuffle $\pi$ as an increasing function from $\{\ell +1,\cdots, 2\ell\}$ to $\{1,2, \cdots 2\ell \}$. Let us define $\theta^{\pi}:\{1,2,\cdots, 2\ell\}\rightarrow \{\pm \}$ by $\theta^{\pi}|_{Im\,\pi}=-$ and $\theta^{\pi}|_{({Im\,\pi})^c}=+$. For any $\pi\in S_{\ell,\ell}$ define
\begin{equation}\label{varphipi}
\varphi_{\pi}(a_0,a_1,\cdots a_{2\ell}):=\int_h a_0(\del^{\theta^{\pi}_1} a_1) (\del^{\theta^{\pi}_2} a_2)\cdots(\del^{\theta^{\pi}_{2\ell}} a_{2\ell}).
\end{equation}
Here $\del^+=\del$, $\del^-=\delb$ and $\theta^{\pi}_i=\theta^{\pi}(i)$. Now suppose that $\pi$ and $\pi'$ are two shuffles that are just different in their values on a single value $i$ such that $|\pi'(i)-\pi(i)|=1$. We define a cochain $\psi_{\pi,\pi'}$ by
\begin{align*}
\psi_{\pi,\pi'}(a_0,a_1,a_2,\cdots,a_{2\ell-1}):=
\int_h a_0(\del^{\theta^{\pi}_1} a_1) (\del^{\theta^{\pi}_2} a_2)\cdots(\del^{\theta^{\pi}_j}\del^{\theta^{\pi'}_j} a_{j})(\del^{\theta^{\pi}_{j+2}} a_{j+1})\cdots(\del^{\theta^{\pi}_{2\ell}} a_{2\ell-1}).
\end{align*}Here $j=min\{\pi(i),\pi '(i)\}$.
It is then easy to prove that $b_{\sigma}\psi_{\pi}=\pm (\varphi_{\pi}-\varphi_{\pi '})$.
The proof is based on the following easy observation.
\[
\del\delb(ab)=a\del\delb b+\del a\delb b-\delb a\del b+(\del\delb a) b.
\]
The term $\del^{\theta^{\pi}_j}\del^{\theta^{\pi'}_j}$ is either $\del\delb$ or $\delb\del$ simply because of our choice of $\pi$ and $\pi '$.

Now we recall an easy combinatorial fact. The number of permutations of $2\ell$ letters including $\ell$ letter $A$ and $\ell$ letter $B$ is $\binom{2\ell}{\ell}=\frac{(2\ell)!}{\ell !\ell !}$. All permutations can be grouped in two groups and in each group there exists an order on permutations $\{\pi_1,...,\pi_r\}$ and $\{\pi '_1,...,\pi '_r\}$ with $r=\frac{1}{2}\binom{2\ell}{\ell}$, such that $\pi_{i+1}$ (respectively $\pi '_{i+1}$), can be obtained from $\pi_i$ (resp. $\pi'_i$) just with replacing the two letters in the spots $j$ and $j+1$ where $1\leq j\leq r-1$. In addition we can always choose $\pi_1=AA\cdots ABB\cdots B$ and $\pi '_1=BB\cdots BAA\cdots A$. The permutation $\pi_r$ has the above mentioned property with respect to one of $\pi'_i$'s.

Now we come back to the case $\mathbb{C}P^{\ell}_q$. We consider a complex structure $(\Omega^{(\bullet,\bullet)}(\mathcal A),\del,\delb)$ on the $*$-algebra $\mathcal A(\qp^{\ell}_q)$ with $*:\Omega^{(p,q)}\rightarrow \Omega^{(q,p)}$ such that $\delb a^*=(\del a)^*$. We have seen that $\Omega^{(0,1)}=\mathfrak M(\sigma ^{0,1})$, where $\sigma^{0,1}$ restricted to $\mathcal U_q(\mathfrak{su}(\ell))$ is the fundamental representation of $\mathcal U_q(\mathfrak{su}(\ell))$ in $\C ^\ell$ and $\sigma^{0,1}(K_1K_2^2\cdots K_\ell^\ell)=q^{\frac{\ell+1}{2}} I$.
The representation $\sigma^{1,0}$ can be obtained from $\sigma^{0,1}$ by conjugation. Define
\[\del a:= \triangleleft (E_\ell, E_\ell E_{\ell-1},\cdots, E_\ell \cdots E_2E_1), \quad \delb a:=\triangleleft (F_\ell \cdots F_2F_1,\cdots F_\ell F_{\ell-1},\, F_\ell).
\]
For an anti-holomorphic 1-form $\omega=(\omega_1,\omega_2,\cdots,\omega_\ell)$ we define
\[
\omega^*:=(-q\omega_\ell^*,q^2 \omega_{\ell-1}^*,\cdots,(-q)^{\ell-1}\omega_2^*,(-q)^{\ell}\omega_1^*).
\]
The property $\delb a^*=(\del a)^*$ holds simply because
\begin{align*}
(a^* \triangleleft F_\ell F_{\ell-1}\cdots F_i)^*=
a \triangleleft S( F_\ell F_{\ell-1}\cdots F_i)^*=
(-q)^{-(\ell-i+1)}a \triangleleft E_\ell E_{\ell-1} \cdots E_i.
\end{align*}
One can define $*$ on anti-holomorphic forms such that  $(\omega \wedge_q \omega ')^*=(-1)^{deg( \omega) deg( \omega ')}\omega '^*\wedge_q \omega^*
$, then extend it to all holomorphic and anti-holomorphic forms with
$\delb a^*=(\del a)^*$.
Note that we can extend $\wedge_q$ on holomorphic forms as \cite{DD}.
One can see that
\begin{align*}
\del a_1 \del a_2\cdots \del a_\ell \delb a_\ell^* \cdots\delb a_2^* \delb a_1^*
&=\del a_1 \del a_2\cdots \del a_\ell (\delb a_\ell)^* \cdots(\delb a_2)^* (\delb a_1)^*\\
&=-\del a_1 \del a_2 \cdots \del a_\ell (\del a_1 \del a_2\cdots\del a_\ell)^*.
\end{align*}
We will need the following simple lemma for future computations.
\begin{lemma}
For any $a_0,a_1,a_2,
\cdots,a_{2\ell+1}$ $\in \mathcal A(\qp^{\ell}_q)$ the following identities hold:
\[
\int_ha_0(\del a_1\cdots\del a_\ell\delb a_{\ell+1}\cdots\delb a_{2\ell})a_{2\ell+1}=\int_h\sigma (a_{2\ell+1})a_0\del a_1\cdots\del a_\ell\delb a_{\ell+1}\cdots\delb a_{2\ell}.
\]
\end{lemma}
\begin{proof}
The space of $\Omega^{(\ell,\ell)}$ is a rank one free $\mathcal A(\qp^\ell_q)$-module. Let $\omega$ be the central basis element for the space of $\Omega^{(\ell,\ell)}$
and let $\del a_1\cdots\del a_\ell\delb a_{\ell+1}\cdots\delb a_{2\ell}=x\omega$. Then
\begin{align*}
&\int_h\Big\{a_0(\del a_1\cdots\del a_\ell\delb a_{\ell+1}\cdots\delb a_{2\ell})a_{2\ell+1}-\sigma (a_{2\ell+1})a_0\del a_1\cdots\del a_\ell\delb a_{\ell+1}\cdots\delb a_{2\ell}\Big\}\\
=&\int_h(a_0 x\omega a_{2\ell+1}-\sigma(a_{2\ell+1})a_0x\omega)\nonumber\\=&\int_h(a_0 x a_{2\ell+1}\omega-\sigma(a_{2\ell+1})a_0x\omega )&\nonumber\\
=&\,h(a_0 xa_{2\ell+1}-\sigma(a_{2\ell+1})a_0x)=0.\nonumber
\end{align*}
The last equality comes from the twisted property of the Haar state.
\end{proof}
Using $\ud=\partial+\delb$, we have
\[
\tau=\sum_{\pi\in S_{\ell,\ell}}\varphi_{\pi},
\]
where $\varphi_{\pi}$ is given by (\ref{varphipi}).
Let $\pi_1=id$, i.e. $\pi_1$ is the shuffle that keeps every letter at the same spot. Define the Hochschild cocycle
\begin{align}\label{phi}
\varphi:=-2r\varphi_{\pi_1},
\end{align}
where $r=\frac{1}{2}\binom{2\ell}{\ell}$.
\begin{theorem}
The 2$\ell$-cocycle $\varphi$ defined by (\ref{phi}), is a twisted positive Hochschild cocycle and it  is cohomologous to the fundamental twisted cyclic cocycle $\tau$.
\end{theorem}
\begin{proof}
We first verify the twisted cocycle property.
\begin{align*}
\varphi(\sigma(a_0),\sigma(a_1),\sigma(a_2)&,\cdots,\sigma(a_{2\ell}))\\&=2r\int_h \sigma(a_0) \partial \sigma (a_1) \cdots \del\sigma (a_\ell)
\delb \sigma (a_{\ell+1}) \cdots\delb \sigma (a_{2\ell})\nonumber\\
&=2r\int_h K \triangleright (a_0\partial a_1 \cdots \partial a_\ell
\delb a_{\ell+1}\cdots \delb a_{2\ell})\\
&= 2r\,\epsilon(K)\int_h a_0\partial a_1
\cdots\partial a_\ell \delb a_{\ell+1}\cdots \delb a_{2\ell}\nonumber\\
&=\varphi(a_0,a_1,a_2,\cdots,a_{2\ell})\nonumber.
\end{align*}
For positivity one can see that
\begin{align}
\varphi(\sigma (a_0^*)a_0,a_1,a_2,\cdots,a_\ell,a_\ell^*,\cdots,a_2^*,a_1^*)&=-2r
\int_h \sigma(a_0^*)a_0 \del a_1 \del a_2 \cdots\del a_\ell \delb a_\ell^* \cdots\delb a_2^* \delb a_1^*\nonumber\\
&=-2r\int_ha_0 \del a_1 \del a_2 \cdots\del a_\ell \delb a_\ell^* \cdots\delb a_2^* \delb a_1^*a_0^*\nonumber\\
&=2r\int_h(a_0 \del a_1 \del a_2 \cdots\del a_\ell)(a_0 \del a_1 \del a_2 \cdots\del a_\ell)^*.\nonumber
\end{align}
One can take $\del a_i=(v^i_1,v^i_2,\cdots,v^i_\ell)$, then using the multiplication rule of type (1,0) forms (for (0,1) forms c.f. \cite{DD}), we find that
$(a_0 \del a_1 \del a_2 \cdots \del a_3)(a_0 \del a_1 \del a_2 \cdots \del a_3)^*=\mu\mu^*$,
where
\[
\mu=a_0\sum_{\pi\in S_\ell}(-q^{-1})^{||\pi||} v^1_{\pi(1)}v^2_{\pi(2)}\cdots v^\ell_{\pi(\ell)}.
\] Hence
\[
\varphi(\sigma (a_0^*)a_0,a_1,a_2,\cdots,a_\ell,a_\ell^*,\cdots,a_2^*,a_1^*)=2r \,h(\mu\mu^*)\geq 0.
\]
Here we used the positivity of the Haar functional $h$.

Now we would like to find the coefficients $m, k$ such that $m\tau-k\varphi_{\pi_1}=b_{\sigma}\psi$ for a suitable $(2\ell-1)$-cocycle $\psi$. Here we order all $\varphi_{\pi}$'s as explained at the beginning of the section, i.e. we use the order for permutations of $\partial$ and $\delb$ to make two sets $\{\varphi_{\pi_1},\varphi_{\pi_2},...,\varphi_{\pi_r}\}$ and $\{\varphi_{\pi '_1},\varphi_{\pi'_2},...,\varphi_{\pi'_r}\}$, where $r=\frac{1}{2}\binom{2\ell}{\ell}$.
For instance we give the formula for one choice of $\varphi_{\pi_2}$.
\begin{align*}
\varphi_{\pi_2}(a_0,a_1,...,a_{2\ell}):=\int_h a_0\delb a_1\delb a_2\cdots\delb a_{\ell-1}\partial a_{\ell}\delb a_{\ell+1}\partial a_{\ell+2}\cdots\partial a_{2\ell}.
\end{align*}
One can show that there exist $2r-1$ twisted cochains $\psi_{\pi,\pi'}$ such that
\begin{align}\label{cohomolog}
&b_{\sigma}\psi_{\pi_1,\pi_2}=\varphi_{\pi_1}-\varphi_{\pi_2},\nn\\
&b_{\sigma}\psi_{\pi_2,\pi_3}=\varphi_{\pi_2}-\varphi_{\pi_3},\nn\\
&
\quad\vdots\nonumber\\
&b_{\sigma}\psi_{\pi_{r-1},\pi_r}=\varphi_{\pi_{r-1}}-\varphi_{\pi_r},\nn\\
&b_{\sigma}\psi_{\pi_{r},\pi'_k}=\varphi_{\pi_{r}}-\varphi_{\pi'_k},\nn\\
&b_{\sigma}\psi_{\pi'_1,\pi'_2}=\varphi_{\pi'_1}-\varphi_{\pi'_2},\nn\\
&b_{\sigma}\psi_{\pi'_2,\pi'_3}=\varphi_{\pi'_2}-\varphi_{\pi'_3},\nn\\
&
\quad\vdots\nn\\
&b_{\sigma}\psi_{\pi'_{r-1}\pi'_{r}}=\varphi_{\pi'_{r-1}}-\varphi_{\pi'_r}
\end{align}
For instance $\psi_{\pi_1,\pi_2}$ (up to a $\pm$ sign) is defined by
\begin{equation*}
\psi_{\pi_1,\pi_2}(a_0,a_1,...,a_{2\ell-1}):=\int_ha_0\partial a_1...\partial a_{\ell-1}(\partial\delb a_{\ell}) \delb a_{\ell+1}...\delb a_{2\ell-1}.
\end{equation*}
Define
\[
\psi:=\sum_{i=1}^{r-1}x_i\psi_{\pi_i,\pi_{i+1}}+x_r\psi_{\pi_i,\pi'_{k}}+\sum_{i=1}^{r-1}x_{r+i}\psi_{\pi'_i,\pi'_{i+1}},
\]
 with constants $x_i$'s $i=1,2,\cdots,2r-1$ have to be determined. We find the following linear system of equations for $m\tau-k\varphi_{\pi_1}=b_{\sigma}\psi$.
\begin{align*}
  \begin{cases}
  m-k-x_1=0\\
  m+x_1-x_2\\
\vdots\\
m+x_{r-1}-x_{r}=0\\
m+x_{r+1}=0\\
m+x_{r+1}-x_{r+2}=0\\
\vdots\\
m+x_{r+k-1}-x_{r+k}=0\\
m+x_{r}-x_{r+k-1}-x_{r+k}=0\\
m+x_{r+k}-x_{r+k+1}=0\\
\vdots\\
m+x_{2r-2}-x_{2r-1}=0\\
m+x_{2r-1}=0
\end{cases}
\end{align*}
This system has the one parameter family of solutions given by
\[x_i=-(2r-i)m \quad \text{for} \quad i\in\{1,2,\cdots, 2r-1\}-\{r+1\},\quad x_{r+1}=-m, \quad k=2rm.
\]
For $m=1$, we have $\tau-2r\varphi_1=b_{\sigma}\psi$.
Note that $\psi_i$'s are defined up to sign.
\end{proof}

\section{The Riemann-Roch theorem for $\qp^\ell_q$, $\ell=1,2$}
First recall that, for classical projective space $\mathbb CP^n$, its sheaf (or equivalently Dolbeault) cohomology with coefficients in the sheaf of holomorphic sections of line bundles $\mathcal O(m)$ are given by
\begin{align*}
H^i(\mathbb CP^n,\mathcal O(m))=\begin{cases}
\mathbb C[z_0,z_1,...,z_n]_m&\text{if}\quad i=0,\,m\geq 0,\\
0&\text{if}\begin{cases}i=0,\,m<0\\0<i<n\\i=n,\, m>-n-1\end{cases}\\
H^0(\mathbb CP^n, \mathcal O(-m-n-1))^*&\text{if}\quad i=n,\,m\leq-n-1.
\end{cases}
\end{align*}
Therefore for the holomorphic Euler characteristic of $\mathcal O(m)$, we get
\begin{align*}
\chi (\qp^1,\mathcal O(m)):&=\text{dim}\, H^0(\qp^1,\mathcal O(m))-\text{dim}\, H^1(\qp^1,\mathcal O(m))=m+1.
\end{align*}
\subsection{ The case of $\qp^1_q$}
This last formula has an analog in the case of $\qp^1_q$. The zeroth cohomology has been computed in \cite{KLS}, but for completeness we recall it here again.
First let us recall that finite dimensional irreducible representations of $U_q(\mathfrak{su}(2))$ are given by vector spaces $V_l$, where $2l\in\mathbb N$ with basis $|l,m\rangle,\, m\in\{-l,...,l\}$. The action on generators are given by
\begin{align*}
K|l,m\rangle&=q^m|l,m\rangle,\\
E|l,m\rangle&=\sqrt{[l-m+1][l+m]}\,|l,m-1\rangle,\\
F|l,m\rangle&=\sqrt{[l+m+1][l-m]}\,|l,m+1\rangle.
\end{align*}
We will have the isomorphism $ A(SU_q(2))=\bigoplus V_l\otimes V_l^*$ and under this isomorphism
the space of canonical quantum line bundle $ L_N:=\{a\in A(SU_q(2))|\quad h\triangleright a=q^{N/2}a\}$ corresponds to $\{|l,N/2\rangle\otimes|l,m\rangle|\quad l\geq|N/2|,\, m=-2l,...,2l\}$. From now on we will use the notation $|l,n,m\rangle=|l,n\rangle\otimes|l,m\rangle$.

The anti-holomorphic part of the connection on  $L_N$ is given by $\nabla^{ \delb}|l,\frac{N}{2},n\rangle:=E|l,\frac{N}{2},n\rangle$. Consider the Dolbeault complex of $\qp^1_q$
\begin{equation}
0\rightarrow L_N\rightarrow\Omega^{(0,1)}\otimes L_N\rightarrow0,\nonumber
\end{equation}
or equivalently
\begin{equation}
0\rightarrow L_N\rightarrow L_{N-2}\rightarrow 0.\nonumber
\end{equation}

One can easily see that $\nabla^{\delb}\xi=E|l,\frac{N}{2},m\rangle=\sqrt{[l-\frac{N}{2}+1][l+\frac{N}{2}]}|l,\frac{N}{2}-1,m\rangle$. To find the holomorphic Euler characteristic $\chi(\qp^1_q,L_N)$, we will consider the following three cases.
\\\\
$\bullet$ {$N\geq 2$}.\\\\
In this case, the kernel of $\nabla^{\delb}$ is zero, simply because $l+\frac{N}{2}$ cannot be zero and $l-\frac{N}{2}+1$ is zero only if $l=\frac{N}{2}-1$, which is impossible in this case, since by assumption $l\geq\frac{N}{2}$.
The Image of $\nabla^{\delb}$ will be generated by the basis elements $|l,\frac{N}{2}-1,m\rangle$ with $l\geq\frac{N}{2}$.
But it differs from basis of $ L_{N-2}$ by elements $|\frac{N}{2}-1,\frac{N}{2}-1,m\rangle$ which can be counted as $N-1$
elements.\\\\
$\bullet$ {$N=1$}.\\\\
Here we have $\nabla^{\delb}\xi=\sqrt{[l-\frac{1}{2}+1][l+\frac{1}{2}]}|l,\frac{1}{2}-1,m\rangle$.
So $E|l,\frac{1}{2},m\rangle=[l+\frac{1}{2}]|l,-\frac{1}{2},m\rangle$ and it is not hard to see that Im$\nabla^{\delb}=L_{N-2}$. The same argument as case $N\geq 2$ shows that Ker$\nabla^{\delb}=0$. Hence $\chi(\qp^1_q, L_N)=0$.\\
\\
$\bullet$ $N\leq 0$.\\\\
If $N\leq0$, $l+\frac{N}{2}=0$ when $l=-\frac{N}{2}$ and this gives the set $\{|-\frac{N}{2},\frac{N}{2},m\rangle|\,m=\frac{N}{2},\frac{N}{2}+\frac{1}{2},...,-\frac{N}{2}\}$ as a basis for the space of holomorphic sections of $\mathcal L_N$.
So dim Ker $\nabla^{\delb}=|N|+1$. In a similar manner to case $N=1$ one can show that the map $\nabla^{\delb}$ is surjective. Therefore we will come to the following result
\begin{align*}
\chi(\qp^1_q, L_N)=-N+1.
\end{align*}
Note that there is a switch between $N$ and $-N$ with respect to the classical case.

\subsection{Serre duality for $\qp^2_q$}
There exists a non-degenerate pairing $\langle\,,\rangle: L_N \times L_{-N}\rightarrow \mathbb {C}$, given by
\begin{equation}
\langle \xi,\eta\rangle:=h(\xi\eta),\quad \forall \xi\in L_N,\quad\forall \eta\in L_{-N}.
\end{equation}
Here $h$ is the Haar state of the quantum group $\mathcal A(SU_q(3))$. The map is obviously bilinear and the nondegeneracy comes from the facts that $L_N^*\subset L_{-N}$ and $h$ is faithful. Now consider the $(0,q)$-Dolbeault complex of $\qp^2_q$
\begin{equation}
0\rightarrow L_N\rightarrow\Omega^{(0,1)}\otimes L_N\rightarrow\Omega^{(0,2)}\otimes L_N\rightarrow 0.
\end{equation}
We would like to state an analogue of Serre duality theorem for this complex as
\begin{proposition}
There exists a non-degenerate pairing defined by
\begin{align*}
&\langle\,,\,\rangle:H^2(\nabla, L_N)\times H^0(\nabla,L_{-N-3})\rightarrow\mathbb {C}\\
&\langle[\xi],[\eta]\rangle:=h(\xi\eta), \quad\forall \xi\in L_{N+3},\quad\forall\eta\in L_{-N-3}.
\end{align*}
\end{proposition}
\begin{proof}
First note that $H^2(\nabla,L_N)$ is a quotient of $L_{N+3}$ and $H^0(\nabla,L_{-N-3})$ is a subspace of $ L_{-N-3}$. We show that this map is well defined. For this, suppose that $\xi$ and $\xi'$ are in the same cohomology class. Hence $h(\xi\eta)-h(\xi'\eta)=h((\xi-\xi')\eta)=h(\delb\alpha\eta)=h(\delb(\alpha\eta)-\alpha\delb \eta)=0$, by noting that $\eta\in \text{Ker}\delb$ and $h$ has invariance property with respect to the map $\delb$. Now non-degeneracy is obvious by the above discussion.\\
\end{proof}
The above result easily can be lifted to the general case of $\qp^\ell_q$ in the following way.
 The pairing
\begin{equation}
\langle \xi,\eta\rangle:=h(\xi\eta),\quad \forall \xi\in L_N,\,\forall \eta\in  L_{-N}.
\end{equation}
is a nondegenerate pairing and hold true passing to the cohomology
\begin{align*}
&\langle\,,\,\rangle:H^l(\nabla, L_N)\times H^0(\nabla,L_{-N-\ell-1})\rightarrow\mathbb {C}\\
&\langle[\xi],[\eta]\rangle:=h(\xi\eta), \quad\forall \xi\in L_{N+\ell+1},\quad\forall\eta\in L_{-N-\ell-1}.
\end{align*}

In the following we will compute the $(0,q)$-Dolbeault cohomology of $\qp^2_q$. The result is analog of the classical case. i.e.
\begin{theorem}
With the above notations
\begin{align*}
H^i(\nabla^{\delb}, L_N)=
\begin{cases}
\mathbb {C}\langle z_1,z_2,z_3\rangle_N\,&\text{if}\quad i=0,\,N\geq 0,\\
0\, &if
\begin{cases}i=0,\,N<0\\i=1,\, N=0\\i=2,\,N>-3\end{cases}\\
\mathbb {C}\langle z_1,z_2,z_3\rangle_{-N-3}^*\,&\text{if}\quad i=2,\,N\leq-3.
\end{cases}
\end{align*}
\end{theorem}
\begin{proof}
The zeroth-cohomology has been computed in \cite{KM} and the second cohomology comes from the Serre duality. So we just have to prove that the triviality of the first cohomology.
In order to do so, we will calculate the Im $\delb_1$ and the Ker $\delb_2$ and show the equality.
\begin{align*}
\delb_1(t(n,n+N)^0_{\underline j})&=(E_1E_2\blacktriangleright  t(n,n+N)^0_{\underline j},E_2\blacktriangleright  t(n,n+N)^0_{\underline j})\\
&=(t(n,n+N)^{1,0,1/2}_{\underline j},t(n,n+N)^{1,0,-1/2}_{\underline j})
\end{align*}
For Ker $\delb_2$ we will use the $\delb_2(v_+,v_-)=-E_2\blacktriangleright v_+-E_2E_1+2[2]^{-1}E_1E_2\blacktriangleright v_-$
.  Applying $v_+=t(n,n+3)^{1,0,1/2}_{\underline j}$ and $v_-=t(n,n+3)^{1,0,-1/2}_{\underline j}$ we will have
\begin{align*}
&-E_2\blacktriangleright t(n,n+3)^{1,0,1/2}_{\underline j}=-\sqrt{\frac{[n][n+3+2]}{[2][3]}}t(n,n+3)^{1,1,0}_{\underline j}
-\sqrt{\frac{[n+2][n+3]}{[2]}}t(n,n+3)^{\underline 0}_{\underline j},\\
&-E_2E_1\blacktriangleright t(n,n+3)^{1,0,-1/2}_{\underline j}=\\&-E_2\blacktriangleright t(n,n+3)^{1,0,1/2}_{\underline j}=-\sqrt{\frac{[n][n+3+2]}{[2][3]}}t(n,n+3)^{1,1,0}_{\underline j}
-\sqrt{\frac{[n+2][n+3]}{[2]}}t(n,n+3)^{\underline 0}_{\underline j},\\&\text{and}\\\\
&2[2]^{-1}E_1E_2\blacktriangleright t(n,n+3)^{1,0,-1/2}_{\underline j}=\\
&2[2]^{-1}E_1\blacktriangleright(\sqrt{[2]}\sqrt{\frac{[n][n+3+2]}{[2][3]}}t(n,n+3)^{1,1,-1}_{\underline j})=2\sqrt{\frac{[n][n+5]}{[2][3]}}
t(n,n+3)^{1,1,0}_{\underline j}
\end{align*}
Hence
\begin{align*}
\delb_2(t(n,n+3)^{1,0,1/2}_{\underline j},t(n,n+3)^{1,0,-1/2}_{\underline j})=-2\sqrt{\frac{[n+2][n+3]}{[2]}}
t(n,n+3)^{\underline 0}_{\underline j}
\end{align*}
This shows that $H^1=\frac{Ker \, \delb_2}{Im \, \delb_1}=0$ in the case of $N=0$.
\end{proof}
By a similar but lengthier calculation, one can prove that $H^0(\nabla^{\delb},L_N)=0$ for all $N\neq0$.

\subsection*{Acknowledgments}
We are much obliged and thankful to
Francesco D'Andrea for kindly and promptly  answering many questions
about the subject of \cite{DL,DDL} and suggesting improvements on the first draft of the current  paper.


Department of Mathematics, University of Western Ontario, London, Ontario, N6A5B7, Canada.

\textit{Email}: masoud@\,uwo.ca\\\\
Department of Mathematics, University of Western Ontario, London, Ontario, N6A5B7, Canada.

\textit{Email}: amotadel@\,uwo.ca
\end{document}